\documentclass[12pt, a4paper]{article}
\usepackage{epsfig}
\usepackage{color} 
\usepackage{latexsym}
\usepackage{amsfonts}
\usepackage{amsmath, amssymb, amsthm}
\usepackage{graphicx}

\newtheorem{theorem}{Theorem}[section]
\newtheorem{proposition}[theorem]{Proposition}
\newtheorem{lemma}[theorem]{Lemma}

\newtheorem{conjecture}[theorem]{Conjecture}

\title{\bf The chromatic number of a signed graph}

\author{Edita M\'a\v cajov\'a\\
\small Department of Computer Science\\[-0.8ex]
\small Comenius University\\[-0.8ex]
\small Mlynsk\'a dolina, 842 48 Bratislava, Slovakia\\
\small\tt macajova@dcs.fmph.uniba.sk\\
\and
Andr\'e Raspaud\\
\small LaBRI\\[-0.8ex]
\small Universit\'e de Bordeaux\\[-0.8ex]
\small 33405 Talence Cedex, France\\
\small\tt raspaud@labri.fr\\
\and
Martin \v Skoviera\\
\small Department of Computer Science\\[-0.8ex]
\small Comenius University\\[-0.8ex]
\small Mlynsk\'a dolina, 842 48 Bratislava, Slovakia\\
\small\tt skoviera@dcs.fmph.uniba.sk
}

\begin{document}

\maketitle

\begin{abstract}
In 1982, Zaslavsky introduced the concept of a proper vertex
colouring of a signed graph $G$ as a mapping $\phi\colon
V(G)\to \mathbb{Z}$ such that for any two adjacent vertices $u$
and $v$ the colour $\phi(u)$ is different from the colour
$\sigma(uv)\phi(v)$, where is $\sigma(uv)$ is the sign of the
edge $uv$. The substantial part of Zaslavsky's research
concentrated on polynomial invariants related to signed graph
colourings rather than on the behaviour of colourings of
individual signed graphs. We continue the study of signed graph
colourings by proposing the definition of a chromatic number
for signed graphs which provides a natural extension of the
chromatic number of an unsigned graph. We establish the basic
properties of this invariant, provide bounds in terms of the
chromatic number of the underlying unsigned graph, investigate
the chromatic number of signed planar graphs, and prove an
extension of the celebrated Brooks theorem to signed graphs.

\end{abstract}

\section{Introduction}
This paper continues the study  of vertex colourings of signed
graphs initiated by Zaslavsky in three seminal papers
\cite{Zas-coloring, Zas-invariants, Zas-colorful} from early
1980's. A signed graph is a graph in which each edge is
labelled with $+1$ or $-1$. The idea of how to colour such a
graph is fairly straightforward: (1) use signed colours so as
to enable vertex-switching, and (2) do so in such a way that
the usual rule to colour adjacent vertices with different
colours is respected as long as the edge that joins them is
positive. The following definition, taken from
\cite{Zas-coloring}, incorporates this idea: Given a signed
graph $G$, a \emph{proper vertex colouring} of~$G$, or simply a
\textit{colouring}, is a mapping $\phi\colon V(G)\to\mathbb{Z}$
such that for each edge $e=uv$ of $G$ the colour $\phi(u)$ is
distinct from the colour $\sigma(e)\phi(v)$, where $\sigma(e)$
denotes the sign of~$e$. In other words, the colours of
vertices joined by a positive edge must not coincide while
those joined by a negative edge must not be opposite.

This definition is natural for several reasons extensively
discussed in \cite{Zas-coloring, Zas-invariants, Zas-colorful}.
In particular, colourings defined in this manner are well
behaved under switching. Recall that \textit{switching} of a
signed graph at a vertex $v$ reverses the sign of each non-loop
edge incident with $v$. The switching operation does not
essentially change the signed graph, because it preserves the
sign product on each circuit. If we switch a properly coloured
signed graph at some vertex, the colouring has to be switched
together with the signature, that is to say, the colour at the
vertex must be replaced with its negative. It is easy to see
that the result of a vertex switching is again a proper
colouring.
Furthermore, a colouring of a \textit{balanced}
signed graph, one where the sign product on each simple circuit
is positive, exactly corresponds to the usual vertex colouring
of the underlying unsigned graph.

Our paper focuses on measuring the complexity of a signed graph
colouring by means of a chromatic number. One possible approach
to introducing the chromatic number of a signed graphs was
proposed by Zaslavsky in \cite[p.~218]{Zas-coloring} where a
\textit{signed colouring} of a signed graph $G$ \textit{in $k$
colours}, or \textit{in $2k+1$ signed colours}, is defined to
be a mapping $V(G)\to \{-k, -(k-1), \ldots, 0. \ldots, k-1,
k\}$; a colouring is \textit{zero-free} if it never assumes the
value $0$. The necessity of treating zero-free colourings
separately derives from a different behaviour of switching with
respect to vertices coloured zero. This led Zaslavsky
\cite[p.~219]{Zas-coloring} to define the \textit{chromatic
polynomial} $\chi_G(\lambda)$ of a signed graph $G$ to be the
function defined for odd positive arguments $\lambda=2k+1$
whose value equals the number of proper signed colourings of
$G$ in $k$ colours. The \textit{balanced chromatic polynomial}
$\chi_G^{\mathrm{b}}(\lambda)$, defined for even positive
arguments $\lambda =2k$, is the function which counts the
zero-free proper signed colourings in $k$ colours. Finally, the
\textit{chromatic number} $\gamma (G)$ of $G$, according to
Zaslavsky \cite[p.~290]{Zas-invariants}, is the smallest
nonnegative integer $k$ for which $\chi_G(2k+1)>0$, and the
\textit{strict chromatic number} $\gamma^*(G)$ of $G$ is the
smallest nonnegative integer $k$ such that
$\chi_G^{\mathrm{b}}(2k)>0$.

The disadvantage of Zaslavsky's definitions is that neither of
these two varieties of chromatic number is a direct extension
of the usual chromatic number of an unsigned graph. This is
because, roughly speaking, the chromatic numbers $\gamma$ and
$\gamma^*$ only count the absolute values of colours. However,
it seems natural to require a signed version of any graph
invariant to agree with its original unsigned version on
balanced signed graphs. We therefore diverge from the
definitions adopted by Zaslavsky in \cite{Zas-coloring,
Zas-invariants, Zas-colorful} and propose the following.

First, we define, for each $n\ge 1$, a set
$M_n\subseteq\mathbb{Z}$ by setting $M_n=\{\pm1, \pm2,\ldots,
\pm k\}$ if $n=2k$, and $M_n=\{0,\pm1, \pm2,\ldots, \pm k\}$ if
$n=2k+1$. A proper colouring of a signed graph $G$ that uses
colours from $M_n$ will be called an \textit{$n$-colouring}.
Thus, an $n$-colouring of a signed graph uses at most $n$
distinct colours. Note that if $G$ admits an $n$-colouring,
then it also admits an $m$-colouring for each $m\ge n$.  The
smallest $n$ such that $G$ admits an $n$-colouring will be
called the \textit{chromatic number} of $G$ and will be denoted
by $\chi(G)$. It is easy to see that the chromatic number of a
balanced signed graph coincides with the chromatic number of
its underlying unsigned graph. Moreover,
$$\chi(G)=\gamma(G)+\gamma^*(G).$$

The aim of this paper is to prove several basic results
concerning the chromatic number of a signed graph. Our results
are divided into three sections. In Section~\ref{sec:basic} we
investigate relationships between the chromatic number of a
signed graph and various graph invariants. Among other things,
we present bounds on the chromatic number of a signed graph by
means of the chromatic number, the acyclic chromatic number,
and the arboricity of the underlying unsigned graph. Our main
result is a Brooks-type theorem for signed graphs, which will
be proved in Section~\ref{sec:Brooks}. The theorem states that
the chromatic number of every simple signed graph $G$ is
bounded by $\Delta(G)$, the maximum degree of $G$, unless $G$
is a balanced complete graph, a balanced odd circuit, or an
unbalanced even circuit. Finally, in Section~\ref{sec:planar},
we deal with the chromatic number of planar signed graphs. We
prove that the chromatic number of every simple signed planar
graph is at most $5$ and make a conjecture that this bound can
be improved to $4$.

\section{Basic properties of signed chromatic number}
\label{sec:basic}

We assume that the reader is familiar with the basic concepts
of signed graph theory such as balance, switching, switching
equivalence, etc. For more information about signed graphs we
refer the reader to \cite{Zas-signed} or to \cite{Zas-coloring,
Zas-invariants, Zas-colorful}.

Recall that a colouring of a signed graph $G$ is a function
$\phi\colon V(G)\to\mathbb{Z}$ such that for each edge $e=uv$
of $G$ the colour $\phi(u)$ is distinct from the colour
$\sigma(e)\phi(v)$, where $\sigma(e)$ is the sign of~$e$. It
follows that in order for $G$ to admit a proper colouring, $G$
cannot have a positive loop. Throughout the rest of this paper
we therefore forbid positive loops in all our signed graphs.
Negative loops and parallel edges are not excluded,
nevertheless, a negative loop at a vertex $v$ forbids
$\phi(v)=0$, and a pair of differently signed parallel edges
between vertices $u$ and $v$ forbids $|\phi(u)|=|\phi(v)|$. In
general, colour $0$ has a different behaviour from other
colours because $-0=0$. Somewhat surprisingly, however,
colourings of simple signed graphs behave in a very much
similar way as colourings of unsigned graphs.

Before proceeding further we would like to remark that
throughout the paper the following argument will be used
without mention: if a signed graph resulting from a series of
switchings is $n$-colourable, then so is the original signed
graph.

We start our investigation by comparing the chromatic number of
a signed graph to the chromatic number of its underlying graph.
In order to make a clear distinction between a signed graph and
its underlying graph, we will use $\underline{G}$ to denote the
underlying graph of a signed graph $G$.

\begin{theorem}
For every loopless signed graph $G$ we have $$\chi(G)\le
2\chi(\underline{G})-1.$$ Furthermore, this bound is sharp.
\end{theorem}

\begin{proof}
Every colouring of $\underline{G}$ with colours in the set
$\{0, 1,\ldots, n-1\}$ is also a signed colouring of $G$
irrespectively of the signature. For $n=\chi(\underline{G})$ we
get a signed colouring of $G$ with $2\chi(\underline{G})-1$
colours from the set $\{0,\pm 1, \ldots, \pm (n-1)\}=M_{2n-1}$,
implying the required inequality.

For the second part, we display an infinite sequence
$(G_n)_{n\ge 1}$ of signed graphs such that $\chi(G)=
2\chi(\underline{G})-1$. To construct $G_n$, we take one
all-positive copy $H_1$ of $K_n$ and $n-1$ all-negative copies
$H_2, H_3, \ldots, H_n$ of $K_n$. For $1\le i\le n$ we denote
the vertices of $H_i$ by $v_{i,1},v_{i,2},\ldots,v_{i,n}$ and
call any two vertices $v_{i,k}$ and $v_{j,k}$ from different
subgraphs $H_i$ and $H_j$ \emph{corresponding}. To finish the
construction, we insert a positive edge between any pair of
non-corresponding vertices from different copies of $K_n$.

Observe that the assignment $\phi(v_{i,j})=j$ for each
$i,j\in\{1,2,\ldots, n\}$ defines a proper colouring of
$\underline{G_n}$. This colouring is optimal since
$\underline{G_n}$ contains a copy of~$K_n$. Hence
$\chi(\underline{G_n})=n$, and by the first part of the proof,
$\chi(G_n)\le 2n-1$. We now show that the chromatic number of
$G_n$ is $2n-1$. Assume, to the contrary, that $G_n$ has a
colouring with colours from the set $M_{2n-2}$. Since $H_1$ is
a balanced complete graph on $n$ vertices, $n$ different
colours $c_1, c_2, \ldots, c_n$ have to be used for the
vertices of $H_1$. We may assume that the vertex $v_{1,i}$ has
colour~$c_i$. Let $D=\{d_1,d_2,\ldots,d_{n-2}\}$ be the set of
all remaining colours, so that
$M_{2n-2}=\{c_1,c_2,\ldots,c_n,\penalty0
d_1,d_2,\ldots,d_{n-2}\}$.

There are $n$ vertices in each $H_i$ with $i\ge 2$ but only
$n-1$ distinct absolute values in $M_{2n-2}$, so there must be
two vertices in each such $H_i$ that receive colours with the
same absolute value. Since the edges within each $H_i$ with
$i\ge 2$ are all negative, such a pair of vertices must receive
the same colour. Further, any vertex $v_{i,j}$ of $H_i$ with
$i\ge 2$ receives a colour from the set $\{c_j\}\cup D$. It
follows that the repeated colour in each $H_i$ is contained in
$D$. If a colour $d\in D$ is repeated in some $H_i$, then it
does not occur at all in any other subgraph $H_j$ with $j\ge
2$: indeed, a vertex $v$ in $H_j$ with colour $d$ would be
joined to at least one of the vertices coloured $d$ in $H_i$ by
a positive edge, which is impossible. Thus each $H_i$ with
$i\ge 2$ has a different repeated colour from $D$. However,
there are $n-1$ all-negative subgraphs $H_i$ in $G_n$ while
there are only $n-2$ elements in $D$, proving that a colouring
of $G_n$ with colours from $M_{2n-2}$ is not feasible.
\end{proof}

A signed graph is \textit{antibalanced} if the sign product on
every even circuit is $+1$ and the sign product on every odd
circuit is $-1$. The signature of an antibalanced signed graph
is easily seen to be switching equivalent to the all-negative
signature. Furthermore, it is well known \cite{Harary} that a
signed graph is antibalanced if and only if its vertex set can
be partitioned into two sets (either of which may be empty)
such that each edge between the sets is positive and each edge
within a set is negative. A balanced antibalanced signed graph
is necessarily bipartite, so antibalanced signed graphs are a
natural generalisation of bipartite graphs. The next result is
an obvious extension of the familiar characterisation of
bipartite graphs.

\begin{proposition}\label{prop:antibal}
A signed graph is $2$-colourable (that is, $\chi(G)\le2$) if
and only if it is antibalanced.
\end{proposition}

\begin{proof}
Let $G$ be a signed graph. If $G$ is antibalanced, we can
switch the signature of $G$ to make it all-negative and assign
$1$ to all the vertices. This produces a $2$-colouring of $G$,
so $\chi(G)\le2$. For the converse assume that $G$ admits a
$2$-colouring. The colour classes form a partition into sets
$V_1$ and $V_{-1}$ which has the property that each edge within
the sets is negative and each edge between the sets is
positive. By the result of \cite{Harary} mentioned above, $G$
is antibalanced.
\end{proof}

A graph $H$ is called \textit{$k$-degenerate} if every subgraph
of $H$ has a vertex of degree at most~$k$.  It is well known
and easy to see that a graph is $k$-degenerate if and only if
there is an ordering $v_1, v_2,\ldots, v_n$ of its vertices
such that for every $1<i\le n$ the vertex $v_i$ has at most $k$
neighbours in $\{v_1,\ldots, v_{i-1}\}$.

\begin{proposition}\label{prop:k-degen}
If a signed graph $G$ is $k$-degenerate, then $\chi(G)\leq
k+1$.
\end{proposition}

\begin{proof}
Order the vertices of $G$ into a sequence $v_1,v_2, \ldots,v_n$
in such a way that for each $i$ with $1<i\le n$ the vertex
$v_i$ has at most $k$ neighbours within $\{v_1,v_2,
\ldots,v_{i-1}\}$. Now, colour the vertices \textit{greedily}
in the described order, that is, at each step assign the vertex
$v_i$ an available colour with the smallest absolute value.
Each coloured neighbour $v$ of $v_i$ forbids one colour to
$v_i$: the colour of $v$ if the edge $vv_i$ is positive, or the
negative of this colour if the edge $vv_i$ is negative. Hence,
having $k+1$ colours from $M_{k+1}$ at our disposal, we can
colour each $v_i$  and produce a proper $(k+1)$-colouring
of~$G$.
\end{proof}

For the next result we need to assume that the signed graph in
question is simple. Recall that the \textit{vertex arboricity}
of a graph $H$, denoted by $a(H)$, is the minimum number of
subsets into which $V(H)$ can be partitioned so that each set
induces a forest. Similarly, the \textit{edge arboricity} of a
graph $H$, denoted by $a'(H)$, is the minimum number of forests
into which $E(H)$ can be partitioned; equivalently, it is the
minimum number of acyclic subgraphs whose union covers $E(H)$.

\begin{proposition}\label{prop:basic2}
Let $G$ be a simple signed graph. The following statements hold
true.
\begin{itemize}
\item[{\rm (i)}] If $G$ is $K_4$-minor-free, then
    $\chi(G)\leq 3$.
\item[{\rm (ii)}] If $G$ is a union of two forests (that
    is, if $a'(\underline{G})\le 2$), then $\chi(G)\leq 4$.
\item[{\rm (iii)}] If $a(G)\le k$, then $\chi(G)\leq 2k$.
\end{itemize}
\end{proposition}

\begin{proof}
(i) A graph is $K_4$-minor-free if and only if it is
series-parallel \cite[pp.~172–-174]{BLS}. A series-parallel
graph may be turned into $K_2$ by a sequence two operations:
(1) replacement of a pair of parallel edges with a single edge
that connects their common end-vertices, and (2) replacement of
a pair of edges incident with a vertex of degree $2$ with a
single edge. From this description it is easy to see that every
simple series-parallel graph is $2$-degenerate. The result now
follows from Proposition~\ref{prop:k-degen}.

(ii) Let $F_1$ and $F_2$ be forests whose union is $G$. Without
loss of generality we may assume that the forests are spanning
and edge-disjoint. To colour $G$, first switch the signature of
$G$ to make $F_1$ all-negative. Then assign each vertex of $G$
colour $1$ or $2$ in such a way that adjacent vertices of $F_2$
do not receive the same colour. The end-vertices of each edge
in $F_2$ are now properly coloured. The same is, true for $F_1$
because each edge in $F_1$ is negative and its end-vertices do
not receive opposite colours. Thus we have a proper
$4$-colouring of $G$.

(iii) Let $\{V_1,V_2,\ldots, V_k\}$ be a partition of $V(G)$
such that the subgraph $F_i$ induced by each $V_i$ is a forest.
Take an arbitrary forest $F_i$ and switch vertices within $F_i$
to make it all-negative. Since the forests $F_1, F_2, \ldots,
F_k$ are pairwise disjoint, switching in one forest does not
interfere with switching in other forests. Thus if we assign
colour $i$ to each vertex of $F_i$, we obtain a proper
$2k$-colouring of $\bigcup_{i=1}^k F_i$. Observe that each edge
of $G$ joining different forests receives colours with
different absolute values, so this colouring is also a proper
$2k$-colouring of the entire $G$. This completes the proof.
\end{proof}

The idea from Part (iii) of the previous proof can be slightly
improved to prove the following.

\begin{proposition}\label{prop:acyclic}
For every simple signed graph $G$ we have
$$\chi(G)\le\chi_a(\underline{G})$$
where $\chi_a(\underline{G})$ denotes the acyclic chromatic
number of $\underline{G}$.
\end{proposition}

\begin{proof}
By the definition, the colour classes of an acyclic colouring
are disjoint independent sets of vertices and every pair of
colour classes induces a forest. To prove the result, first
assume that $\chi_a(\underline{G})$ is even, say
$\chi_a(\underline{G})=2k$. Choose an arbitrary acyclic
colouring of~$\underline{G}$ with $2k$ colours from the set
$\{1,2,\ldots, 2k\}$. If we arrange the colours into pairs
$\{2i-1,2i\}$ for $i\in\{1,2,\ldots, k\}$, then we obtain a
partition of $V(G)$ into $k$ sets each inducing a forest. Hence
$a(\underline{G})\le k$, and from
Proposition~\ref{prop:basic2}~(iii) we get $\chi(G)\le
2k=\chi_a(\underline{G})$. To finish the proof assume that
$\chi_a(\underline{G})$ is odd, say
$\chi_a(\underline{G})=2k+1$. In this case we proceed similarly
except that we use one more colour $2k+1$, and recolour the
vertices receiving $2k+1$ from the acyclic colouring with
colour $0$. It is easy to see the result is a proper
$(2k+1)$-colouring of~$G$, again implying that $\chi(G)\le
2k+1=\chi_a(\underline{G})$.
\end{proof}

\section{A Brooks theorem for signed graphs}\label{sec:Brooks}

The purpose of this section is to state and prove a signed
graph version of one of the most fundamental results on graph
colourings, the famous theorem of Brooks \cite{Brooks}. Let $G$
be a signed graph with maximum degree $\Delta$. If we colour
$G$ greedily with respect to any ordering of its vertices,
then, for each vertex in turn at most $\Delta$ colours are
forbidden by its previously coloured neighbours. Therefore
$\chi(G)\le\Delta+1$ just as in the unsigned case. The bound is
clearly reached whenever $G$ is a balanced complete graph or a
balanced odd circuit. Indeed, these two families are familiar
from the classical Brooks theorem as the only connected graphs
whose chromatic number reaches the bound $\Delta+1$. In the
signed case there is one additional extremal class: even
unbalanced circuits. Our main result shows that, among
connected simple signed graphs, these three infinite families
provide the only signed graphs whose chromatic number reaches
the bound $\Delta$+1.

\begin{theorem}\label{thm:Brooks}
Let $G$ be a simple connected signed graph. If $G$ is not a
balanced complete graph, a balanced odd circuit, or an
unbalanced even circuit, then $\chi(G)\le\Delta(G)$.
\end{theorem}

Before proving Theorem~\ref{thm:Brooks} we need several
preparatory results. The first of them concerns the chromatic
numbers of signed complete graphs.

\begin{proposition} \label{prop:complete}
If $G$ is a signed complete graph on $n$ vertices, then
$\chi(G)\le n$. Furthermore, $\chi(G)= n$ if and only if $G$ is
balanced.
\end{proposition}

\begin{proof}
If $G$ is balanced, then clearly $\chi(G)=n$. In the rest of
the proof we therefore assume $G$ to be unbalanced and show
that $\chi(G)\le n-1$.

We first consider signed complete graphs of odd order and
proceed by induction. If $n=3$, then $G$ is switching
equivalent to a circuit of length $3$ with all edges negative.
Putting $1$ to all the vertices defines a $2$-colouring of $G$,
and the conclusion holds. Assume that $n\ge 5$ is odd, say
$n=2d+1$. Since $G$ is unbalanced, there is an unbalanced
triangle $T$ in $G$. Pick any two vertices $x$ and $y$ of $G$
that do not belong to~$T$, and switch the signature of $G$, if
necessary, to make the edge $xy$ negative. The graph
$G-\{x,y\}$ is an unbalanced complete graph on $n-2$ vertices,
so the induction hypothesis guarantees a colouring of
$G-\{x,y\}$ with $n-3$ colours from the set $\{\pm1,
\pm2,\ldots,\pm(d-1)\}$. It is now sufficient to assign $d$ to
both $x$ and $y$, producing a colouring of $G$ with $n-1$
colours from the set $\{\pm1, \pm2,\ldots,\pm d\}$. Thus
$\chi(G)\le n-1$.

For $n$ even, remove a suitable vertex $v$ from $G$ in such way
that $G-v$ is unbalanced; it is easy to see that such a vertex
always exists. By the previous part, we can colour $G-v$ with
$n-2$ colours from the set $\{\pm1, \pm2,\ldots,\pm (n-2)/2\}$.
Assigning $0$ to $v$ yields an $(n-1)$-colouring of $G$, so
$\chi(G)\le n-1$ again.
\end{proof}

We further need two lemmas. The first lemma is a standard tool
for colouring graphs greedily. The second lemma is due to
Lov\'asz and was crucial in his short proof of Brooks' theorem
\cite{Lovasz}. Its proof can also be found in Cranston and
Rabern \cite[Lemma~1]{Cranston}.

\begin{lemma} \label{lemma:order}
The vertices of every connected graph $G$ can be ordered in a
sequence $x_1,\penalty0 x_2,\ldots,x_n$ in such a way that
$x_n$ is any preassigned vertex of $G$ and for $i<n$ each
vertex $x_i$ has a neighbour among $x_{i+1}, x_{i+2}, \ldots,
x_n$.
\end{lemma}

\begin{lemma} \label{lemma:2vertices}
Let $G$ be a $2$-connected graph with $\Delta(G)\ge 3$ other
than a complete graph. Then $G$ contains a pair of vertices $a$
and $b$ at distance $2$ such that the graph $G-\{a,b\}$ is
connected.
\end{lemma}

Now we are ready to prove Theorem~\ref{thm:Brooks}.


\vspace{3mm}

\noindent\textbf{Proof of Theorem~\ref{thm:Brooks}.} If $G$ is
an unbalanced complete graph, then the conclusion follows from
Proposition~\ref{prop:complete}. The conclusion is also true
whenever $G$ is a path, a balanced even circuit, or an
unbalanced odd circuit. Thus we may assume that $G$ is a simple
connected signed graph of order $n$ with maximum degree
$\Delta\ge 3$ which is not complete. We distinguish two cases.

\medskip
Case 1. \textit{The graph $G$ is $2$-connected}. By
Lemma~\ref{lemma:2vertices}, $G$ contains a path $axb$ such
that $a$ is not adjacent to $b$ and $G-\{a,b\}$ is connected.
We switch the signature at $a$ and $b$ in such a way that the
edges $ax$ and $bx$ are both positive. Next, we choose an
ordering  $x_1,x_2,\ldots,x_{n-2}$ of the vertices of
$G-\{a,b\}$ as in Lemma~\ref{lemma:order} with $x_{n-2}=x$. We
now start colouring $G$ with colours from $M_{\Delta}$ by
assigning colour $1$ to both $a$ and~$b$. Then we colour the
vertices $x_1, x_2,\ldots,x_{n-3}$ greedily in the given order.
Each $x_i\ne x$ has a neighbour among its successors in
$G-\{a,b\}$, so $x_i$ has at most $\Delta-1$ neighbours
previously coloured. Each coloured neighbour $w$ forbids one
colour to $x_i$, so at least one colour from $M_{\Delta}$ is
still available for $x_i$, and we can proceed up to $x_{n-3}$.
At last we colour the vertex $x_{n-2}=x$. All the neighbours of
$x$ now have their colours, but since the colours of $a$ and
$b$ are the same, one colour from $M_\Delta$ is available for
$x$. Thus we can complete colouring $G$ with colours from
$M_\Delta$. This establishes Case~1.

\medskip
Case 2. \textit{The graph $G$ has a cut-vertex.} In this case
we proceed by contradiction and assume that $G$ is a signed
graph of minimum order that cannot be $\Delta$-coloured.  We
first derive the following property of~$G$.

\begin{itemize}
\item[(*)] \textit{Each cut-vertex $v$ of $G$ is incident
    with a bridge $vw$ such that $\deg(v)=\deg(w)=\Delta$.
    Furthermore, $G-vw$ is $\Delta$-colourable and in every
    $\Delta$-colouring of $G-vw$ both $v$ and $w$ are
    coloured $0$.}
\end{itemize}

To prove (*), take an arbitrary cut-vertex $v$ of $G$ and let
$V_1,V_2,\ldots, V_s$ be the vertex sets of the components of
$G-v$; let $G_i$ denote the subgraph induced by set $V_i\cup
\{v\}$. By the minimality of $G$, each $G_i$ admits a
$\Delta$-colouring $c_i$. If a certain subgraph $G_j$ has
$\deg_{G_j}(v)\le \Delta-2$, then we can clearly choose $c_j$
to have $c_j(v)\ne 0$. Note that any permutation of non-zero
colours from $M_\Delta$ transforms a proper $\Delta$-colouring
into another proper $\Delta$-colouring. Thus if we had
$\deg_{G_i}(v)\le \Delta-2$ for each $i\in\{1,2,\ldots, s\}$,
we could permute the colours within each $G_i$ in such a way
that $c_i(v)=1$ for every $i\in\{1,2,\ldots,s\}$. The
colourings $c_1, c_2, \ldots, c_s$ then could be easily
combined into a $\Delta$-colouring of the whole of $G$, which
would be a contradiction. Therefore there exists a subgraph
$G_j$ such that $\deg_{G_j}(v)> \Delta-2$. This is only
possible when $\deg_{G_j}(v)=\Delta-1$, $s=2$, and $v$ is
incident with a bridge $vw$. Let $G_v$ and $G_w$ be the
components of $G-vw$ containing $v$ and~$w$, respectively.
Clearly, each of $G_v$ and $G_w$ has a $\Delta$-colouring. If
one of the bridge-ends, say $w$, had a non-zero colour, we
could easily permute the colours within $G_w$ and then combine
the colourings of $G_v$ and $G_w$ into a $\Delta$-colouring of
the entire $G$. Thus for every $\Delta$-colouring $\phi$ of
$G_v$ and every $\Delta$-colouring $\psi$ of $G_v$ we have
$\phi(v)=\psi(w)=0$. This clearly forces
$\deg(v)=\deg(w)=\Delta$ (and $\Delta$ odd), and establishes
(*).

\medskip
To finish the proof, let us choose the cut-vertex $v$ of $G$
and the bridge $vw$ in such a way that $G_v$ is bridgeless.
Take any neighbour $u$ of $v$ in~$G_v$. If $u$ was a
cut-vertex, then by (*) it would be incident with a bridge of
$G$, which is impossible because $G_v$ was chosen bridgeless.
Therefore $G_v-u$ is connected. Lemma~\ref{lemma:order} now
ensures that the vertices of $G_v-u$ can be ordered into a
sequence $x_2,x_3,\ldots,x_m$ with $x_m=v$ in such a way that
for $2\le i<m$ each $x_i$ has a neighbour among its successors.
If we set $x_1=u$, the same becomes true for $1\le i<m$,
because $x_1$ is adjacent to $x_m$. We now assign $0$ to $x_1$
and colour the vertices greedily in the indicated order. The
result is a $\Delta$-colouring of $G_v$ under which $v$
receives a non-zero colour, contradicting (*). This
contradiction establishes the theorem. {\hfill$ \Box$}

\section{Chromatic number of signed planar graphs}
\label{sec:planar}

We conclude this article with a brief discussion of chromatic
numbers of signed planar graphs. Our main result here is a five
colour theorem for signed planar graphs.

\begin{theorem}
For every simple signed planar graph $G$ one has $\chi(G)\le
5$. Furthermore,
\begin{itemize}
\item[\rm (i)] if $G$ is triangle-free, then $\chi(G)\leq
    4$; and
\item[\rm (ii)] if $G$ is has girth at least $5$, then
    $\chi(G)\leq 3$.
\end{itemize}
\end{theorem}

\begin{proof}
Borodin~\cite{Borodin} proved that every planar graph is
acyclically $5$-colourable. If we combine this fact with
Proposition~\ref{prop:acyclic}, we obtain
$\chi(G)\le\chi_a(\underline{G})\le 5$, as claimed. Next, it is
an easy consequence of Euler's formula that every triangle-free
planar graph contains a vertex of degree at most~$3$; hence it
is 3-degenerate. The statement (i) now follows from
Proposition~\ref{prop:k-degen}. Finally, to prove the statement
(ii) we employ  a result of Borodin and Glebov
\cite{BorodinGlebov} that the vertex set of every planar graph
of girth at least $5$ has a partition $\{U,W\}$ where $U$ is an
independent set and $W$ induces a forest. We take the forest
$F$ induced by $W$ and switch the signature of $G$ to make $F$
all-negative. Now we can assign the vertices of $U$ colour $0$
and those of $W$ colour $1$, producing a proper $3$-colouring
of~$G$.
\end{proof}

The four colour theorem implies that the underlying graph of
every signed planar graph $G$ can be properly coloured with
four colours. It is natural to ask whether it is possible to
find a $4$-colouring of $G$ that respects the constraints of
the signature. To this end, we propose the following
conjecture:

\begin{conjecture}
Every simple signed planar graph $G$ has $\chi(G)\le 4$.
\end{conjecture}

\subsection*{Acknowledgements}
The work of the first and the third author was partially
supported from the grant VEGA 1/0474/15.

The authors would like to thank G\'abor Wiener for a
stimulating discussion on the topic of this paper and to Thomas
Zaslavsky for useful suggestions and encouragement to use the term
``the chromatic number of a signed graph'' for the main concept of this paper,
previously termed "the signed chromatic number".


\end{document}